\documentclass[11pt, reqno]{amsart}
\usepackage[margin=3.6cm]{geometry}
\usepackage{array,booktabs,tabularx}
\usepackage{graphicx}
\usepackage{amssymb}
\usepackage{enumerate}
\usepackage{color}

\newtheorem{theorem}{Theorem}

\newtheorem{proposition}[theorem]{Proposition}
\newtheorem{lemma}[theorem]{Lemma}
\newtheorem{Lem}[theorem]{Lemma}

\theoremstyle{definition}
\newtheorem{definition}{Definition}
\newtheorem{remark}{Remark}

\newcommand{\leb}{\lambda}
\newcommand{\W}{\Omega}
\newcommand{\R}{{\mathbb R}}
\newcommand{\Z}{{\mathbb Z}}

\newcommand{\ep}{\varepsilon}

\newcommand{\gives}{\ensuremath{\rightarrow}}
\newcommand{\x}{\ensuremath{\times}}

\newcommand{\dist}{\operatorname{dist}_g}
\newcommand{\Dist}{\operatorname{dist}}

\newcommand{\abs}[1]{\ensuremath{\left| #1 \right|}}
\newcommand{\lr}[1]{\ensuremath{\left(#1 \right)}}
\newcommand{\norm}[1]{\left\lVert#1\right\rVert}
\newcommand{\inprod}[2]{\ensuremath{\left\langle#1,#2\right\rangle}}

\newcommand{\twiddle}[1]{\ensuremath{\widetilde{#1}}}
\newcommand{\Union}{\ensuremath{\bigcup}}

\newcommand{\w}{\omega}
\newcommand{\dell}{\ensuremath{\partial}}
\newcommand{\set}[1]{\ensuremath{\{#1\}}}

\DeclareMathOperator{\vol}{vol}

\DeclareMathOperator{\supp}{supp}
\DeclareMathOperator{\inj}{inj}

\title[Scaling Limit for the Spectral Projector]
{Scaling limit for the kernel of the Spectral Projector and remainder estimates in the pointwise Weyl Law}

\author[Y. Canzani]{Yaiza Canzani}
\author[B. Hanin]{Boris Hanin}

\address[Y. Canzani]{ Institute for Advanced Study and  Harvard University.\medskip}
 \email{canzani@math.ias.edu}
\address[B. Hanin]{ Massachusetts Institute of Technology.\medskip}
\email{bhanin@mit.edu}

\thanks{Y.C. was partially supported by an NSERC Postdoctoral Fellowship and by NSF grant DMS-1128155.  B.H. was partially supported by NSF grant  DMS-1400822.}

\begin{document}
\maketitle


\begin{abstract}
Let $(M,g)$ be a compact smooth Riemannian manifold. We obtain new off-diagonal estimates as $\leb\gives \infty$ for the remainder in the pointwise Weyl Law for the kernel of the spectral projector of the Laplacian onto functions with frequency at most $\leb$. A corollary is that, when rescaled around a non self-focal point, the kernel of the spectral projector onto the frequency interval $(\leb,\leb+1]$ has a universal scaling limit as $\leb\gives \infty$ (depending only on the dimension of $M$). Our results also imply that if $M$ has no conjuage points, then immersions of $M$ into Euclidean space by an orthonormal basis of eigenfunctions with frequencies in $(\leb,\leb+1]$ are embeddings for all $\leb$ sufficiently large. 

\end{abstract}

\section{Introduction}
Suppose that $(M,g)$ is a smooth, compact, Riemannian manifold without boundary of dimension $n\geq 2.$ Let $\Delta_g$ be the non-negative Laplacian acting on $L^2(M,g,\R),$ and let $\set{\varphi_j}_j$ be an orthonormal basis of eigenfunctions: 
\begin{equation}
\Delta_g\varphi_j = \leb_j^2\, \varphi_j,\label{E:Efns}
\end{equation}
with $0=\leb_0^2< \leb_1^2\leq \leb_2^2\leq \cdots$. This article concerns the $\leb \gives \infty$ asymptotics of the Schwartz kernel 
\begin{equation}
E_\leb(x,y)=\sum_{\leb_j\leq \leb}\varphi_j(x)\varphi_j(y)\label{E:Big Ker Def}
\end{equation}
of the spectral projection
\[E_\leb:~~L^2(M,g)\rightarrow \bigoplus_{\mu\in (0, \leb]} \ker\lr{\Delta_g - \mu^2}\]
onto functions with frequency at most $\leb.$ We are primarily concerned with the behavior of $E_\leb(x,y)$ at points $x,y \in M$ for which the Riemannian distance $\dist(x,y)$ is less than the injectivity radius $\inj(M,g)$ so that the inverse of the exponential map $\exp_y^{-1}(x)$ is well-defined. We write
\begin{equation}
E_\leb\lr{x,y}=\frac{\leb^n}{(2\pi)^n}\int_{|\xi|_{g_y}<1} e^{i\leb\langle{\exp_y^{-1}(x)},{\xi}\rangle_{g_y}}\,\frac{d\xi}{\sqrt{\abs{g_y}}} \, +\,R(x,y,\leb),\label{E:LWL}
\end{equation}
where the remainder $R(x,y,\leb)$ is a smooth function of $x,y$. The integral in \eqref{E:LWL} is over the cotangent fiber $T_y^*M,$ and it is coordinate independent because the integration measure $d\xi /\sqrt{\abs{g_y}}$ is the quotient of the natural symplectic form $d\xi dy$ on $T^*M$ by the Riemannian volume form $\sqrt{\abs{g_y}}dy.$ The integral is also symmetric in $x$ and $y$, which can be seen by changing variables from $T_y^*M$ to $T_x^*M$ using the parallel transport operator (cf \eqref{E:Parallel Transport}).

Our main result, Theorem \ref{close to diagonal}, fits into a long history of estimates on $R(x,y,\leb)$ as $\leb \to +\infty$ (cf \S \ref{S: discussion Thm 1} for some background). To state it, we need a definition from \cite{Saf, SZII}. 

\begin{definition}[Non self-focal point] \label{D:aperiodic}
A point $x \in M$ is said to be \emph{non self-focal} if the set of unit covectors 
\begin{equation}
\mathcal L_x=\{\xi \in S_x^*M| \;\;\exists \,t>0 \;\;\; \text{with}\;\; \;\exp_x\lr{t\xi}=x\}\label{E:Loopset Def}
\end{equation}
has zero measure with respect to the surface measure induced by $g$ on $S_x^*M$. 
\end{definition}
\begin{theorem}\label{close to diagonal}
Let $(M,g)$ be a compact, smooth, Riemannian manifold of dimension $n \geq 2$, with no boundary. Suppose $x_0\in M$ is a non self-focal point, and let $r_\leb$ be a non-negative function with $\lim_{\leb \to \infty} r_\leb=0$. Then, 
  \begin{equation} \label{E:Near-Diag Rem Est}
    \sup_{x,y\in B(x_0, r_\leb)} \abs{R(x,y,\leb)}=o(\leb^{n-1}),
\end{equation}
as $\leb \to \infty.$ Here,  $B(x_0,r_\leb)$ denotes the geodesic ball of radius $r_\leb$ centered at $x_0$, and the rate of convergence depends on $x_0$ and $r_\leb$. 
\end{theorem}

The little oh estimate \eqref{E:Near-Diag Rem Est} is not new for $x=y$ (i.e. $r_\leb=0$). Both Safarov in \cite{Saf} and Sogge-Zelditch in \cite{SZ} show that $R(x,x,\leb)=o(\leb^{n-1})$ when $x$ belongs to a compact subset of the diagonal in $M\times M$ consisting only of non self-focal points (see also \cite{SV}). Safarov in \cite{Saf} also obtained $o(\leb^{n-1})$ estimates on $R(x,y,\leb)$ for $(x,y)$ in a compact subset of $M \times M$ that does not intersect the diagonal (under the assumptions of Theorem \ref{T:LWL aperiodic}). Theorem \ref{close to diagonal} simultaneously allows $x\neq y$ and $\dist(x,y)\gives 0$ as $\leb\gives \infty$, closing the gap between the two already known regimes. We refer the reader to \S \ref{S: discussion Thm 1} for further discussion and motivation for Theorem \ref{close to diagonal} and to \S \ref{S: outline main thm} for an outline of the proof.

An elementary corollary of Theorem \ref{close to diagonal} is Theorem \ref{T:Main}, which gives scaling asymptotics for the Schwartz kernel 
\begin{equation}
E_{_{(\leb,\leb+1]}}(x,y):=\sum_{\leb< \leb_j\leq \leb+1}\varphi_j(x)\varphi_j(y) \label{E: def of E}
\end{equation}
of the orthogonal projection 
\[E_{_{(\leb,\leb+1]}}=E_{\leb+1}-E_\leb:~~ L^2(M,g)\rightarrow \bigoplus_{\mu\in (\leb, \leb+1]} \ker\lr{\Delta_g - \mu^2}.\] 
Passing to polar coordinates in \eqref{E:LWL} and using that 
\begin{equation}\label{E:Fourier-Bessel}
\int_{S^{n-1}} e^{i\langle v, \w\rangle} d\w= (2\pi)^{n/2} \frac{J_{\frac{n-2}{2}}(|v|)}{|v|^{\frac{n-2}{2}}},
\end{equation}
it is straight forward to obtain the following result.
\begin{theorem}\label{T:Main}
Let $(M,g)$ be a compact, smooth, Riemannian manifold of dimension $n\geq 2$, with no boundary.  Let $x_0\in M$ be a non self-focal point. 
Consider any non-negative function $r_\leb$ satisfying $r_\leb \gives 0$ as $\leb \gives \infty.$  Then,
\begin{equation}\label{E:Scaling Limit 2}
  \sup_{\substack{x,y\in B(x_0, r_\leb)}}\abs{E_{_{(\leb,\leb+1]}}\left(x,y\right)-\frac{\leb^{n-1}}{(2\pi)^{\frac{n}{2}}} \,\frac{J_{\frac{n-2}{2}}\left(\leb \dist(x,y)\right)}{(\leb \dist(x,y))^{\frac{n-2}{2}}}}=o(\leb^{n-1}),
 \end{equation}
 where $J_\nu$ is the Bessel function of the first kind with index $\nu,$ $B(x_0, r_\leb)$ denotes the geodesic ball of radius $r_\leb$ centered at $x_0,$ and $\dist$ is the Riemannian distance. 
\end{theorem}
\begin{remark}
Under the assumptions of Theorem \ref{T:Main}, relation \eqref{E:Scaling Limit 2} holds for $E_{(\leb,\leb+\delta]}$ with any $\delta>0.$ The difference is that the Bessel function term is multiplied by $\delta$ and that the rate of convergence depends on $\delta.$ Our proof of Theorem \ref{T:Main} is insensitive to the choice of $\delta$.
\end{remark}
In normal coordinates at $x_0,$ \eqref{E:Scaling Limit 2} therefore implies
\begin{equation}\label{E:Scaling Limit 0}
 \sup_{\substack{\abs{u},\abs{v}<r_0}}\abs{E_{_{(\leb,\leb+1]}}\left(x_0+\frac{u}{\leb},x_0+\frac{v}{\leb}\right)- \frac{\leb^{n-1}}{\lr{2\pi}^n}\int_{S^{n-1}}e^{i\inprod{u-v}{w}}d\w}=o(\leb^{n-1})
 \end{equation}
as $\leb\gives \infty$. The measure $d\w$ is the Euclidean surface measure on the unit sphere $S^{n-1},$ and the rate of convergence of the error term depends on $r_0$ and the point $x_0.$ The integral over $S^{n-1}$ in \eqref{E:Scaling Limit 0} is the kernel of the spectral projector onto the generalized eigenspace of eigenvalue $1$ for the flat Laplacian on $\R^n$ (cf \cite{Hel} and \S 2.1 in \cite{Zel4}).  
 
We believe the estimate \eqref{E:Near-Diag Rem Est} holds for any number of covariant derivatives $\nabla_x^j\nabla_y^k$ of the remainder $R(x,y,\leb)$ with $o(\leb^{n-1})$ replaced by $o(\leb^{n-1+j+k}).$ This would immediately imply that the $C^0$ converence in \eqref{E:Scaling Limit 2} can be upgraded to $C^k$ converence for all $k.$ Proving this is work in progress by the authors. Since $E_{(\leb, \leb+1]}$ is the covariance kernel for asymptotically fixed frequency random waves on $M$ (cf. \cite{SW, NS, Zel2}), this $C^\infty$ converence would show that the integral statistics of monochromatic random waves near a non self-focal point depend only on the dimension of $M$. We refer the reader to \S \ref{S: discussion Thm 2} for further discussion and motivation for Theorem \ref{T:Main}.


\subsection{Applications}
Combining Theorem \ref{close to diagonal} with prior results of Safarov in \cite{Saf}, we obtain little oh estimates on $R(x,y,\leb)$ without requiring $x,y$ to be in a shrinking neighborhood of a single non-focal point. We recall the following definition from \cite{Saf, SZII}.
 \begin{definition}[Mutually non-focal points]\label{D:non-focal}
   Let $(M,g)$ be a Riemannian manifold. We say that $x,y\in M$ are \emph{mutually non-focal} if the set of unit covectors 
\begin{equation}
\mathcal L(x,y)=\{\xi \in S_x^*M| \;\;\exists \,t>0 \;\;\; \text{with}\;\; \;\exp_x\lr{t\xi}=y\}\label{E:Loopset Def}
\end{equation}
has zero measure with respect to the Euclidean surface measure induced by $g$ on $S_x^*M$. 
 \end{definition} 

\begin{theorem}\label{T:LWL aperiodic}
Let $(M,g)$ be a compact, smooth, Riemannian manifold of dimension $n\geq 2$, with no boundary. Consider any compact set $K\subseteq M\x M$ such that if $(x,y)\in K,$ then
$x,y$ are mutually non-focal and either $x$ or $y$ is a non self-focal point.
Then, as $\leb \gives \infty,$ we have
  \begin{equation}\label{E:Global Rem Est}
  \sup_{\lr{x,y}\in K} \abs{R(x,y,\leb)}=o(\leb^{n-1}).
\end{equation}
\end{theorem}
\begin{remark}
Theorem \ref{T:LWL aperiodic} applies with $K=M\x M$ if $(M,g)$ has no conjugate points.
\end{remark}

Theorem \ref{T:LWL aperiodic} $-$ proved in in \S \ref{S:Aperiodic Pf} $-$ can be applied to studying immersions of $(M,g)$ into Euclidean space by arrays of high frequency eigenfunctions. Let $\{\varphi_{j_1},\ldots, \varphi_{j_{m_\leb}}\}$ be an orthonormal basis for $\bigoplus_{\leb<\mu\leq \leb+1} \text{ker}(\Delta_g- \mu^2)$  and consider the maps 
$$\Psi_{_{(\leb, \leb+1]}}:M\gives \R^{m_\leb}, \quad\qquad 
\Psi_{_{(\leb,\leb+1]}}(x)=\sqrt{\frac{(2\pi)^n}{2 \leb^{n-1}}}\lr{\varphi_{j_1}(x),\ldots, \varphi_{j_{m_\leb}}(x)}.$$
 The  $\leb^{-\frac{n-1}{2}}$ normalization is chosen so that the diameter of $\Psi_{_{(\leb,\leb+1]}}(M)$ in $\R^{m_\leb}$ is bounded above and below as $\leb \gives \infty.$ Maps related to $\Psi_\leb$ are studied in \cite{Ber, JMS, Pot, Zel2}. 
In particular, Zelditch in \cite[Proposition 2.3]{Zel2} showed that the maps $\Psi_{_{(\leb,\leb+1]}}$ are almost-isometric immersions for large $\leb$ in the sense that a certain rescaling of the pullback $\Psi_\leb^*(g_{euc})$ of the Euclidean metric on $\R^{m_\leb}$ converges pointwise to $g.$  A consequence of Theorem \ref{T:LWL aperiodic} is that these maps are actually embeddings for $\leb$ sufficiently large.
 
 \begin{theorem}\label{Embedding}
Let $(M,g)$ be a compact, smooth, Riemannian manifold of dimension $n \geq 2$, with no boundary. If every point $x\in M$ is non self-focal and all pairs $x,y\in M$ are mutually non-focal, then there exists $\leb_0>0$ so that the maps $\Psi_{_{(\leb,\leb+1]}}: M \to \mathbb R^{m_\leb}$ are embeddings for all $\leb \geq \leb_0$.
\end{theorem}
\noindent We prove Theorem \ref{Embedding} in \S \ref{S:Pf Embedding}. Note that this result does not hold on the round spheres $S^n\subseteq \R^{n+1}$ since even spherical harmonics take on equal values at antipodal points. Since $\Psi_{_{(\leb, \leb+1]}}$ are embeddings for $\leb$ large, it is natural to study $\Psi_{_{(\leb, \leb+1]}}(M)$ as a metric space equipped with the distance, $\Dist_\leb$, induced by the embedding:
\begin{align}
\Dist_\leb^2(x,y):&=\norm{\Psi_{_{(\leb,\leb+1]}}(x)-\Psi_{_{(\leb, \leb+1]}}(y)}_{l^2(\R^{m_\leb})}^2\\
&=\frac{\lr{2\pi}^n}{2\leb^{n-1}}\lr{E_{_{(\leb, \leb+1]}}(x,x)+E_{_{(\leb, \leb+1]}}(y,y)-2E_{_{(\leb, \leb+1]}}(x,y)}\label{E:dist def}
\end{align}
Theorem \ref{T:Dist Function} $-$ proved in \S \ref{S:Dist Function Pf} $-$ gives precise asymptotics for $\Dist_\leb(x,y)$ in terms of $\dist(x,y)$.
\begin{theorem}\label{T:Dist Function}
 Let $(M,g)$ be a compact, smooth,  Riemannian manifold of dimension $n \geq 2$, with no boundary. Suppose further that every $x\in M$ is non self-focal and all pairs $x,y\in M$ are mutually non-focal. As $\leb \gives \infty,$ we have
 \begin{equation}
\sup_{x,y\in M}\abs{\frac{1}{\leb^2\dist^2(x,y)}\left[\Dist_\leb^2(x,y)- \lr{\vol(S^{n-1})-{(2\pi)^{\frac{n}{2}}}\frac{J_{\frac{n-2}{2}}\lr{\leb \dist(x,y)}}{\lr{\leb \dist(x,y)}^{\frac{n-2}{2}}}}\right]}=o\lr{1}.\label{E:Induced Distance}
\end{equation}
\end{theorem}

\subsection{Discussion of Theorem \ref{close to diagonal}} \label{S: discussion Thm 1}
Theorem \ref{close to diagonal} is an extension of H\"ormander's pointwise Weyl law \cite[Theorem 4.4]{Hor}. H\"ormander proved that there exists $\ep>0$ so that if the Riemannian distance $\dist(x,y)$ between $x$ and $y$ is less than $\ep,$ then 
\begin{equation}\label{E:H LWL}
E_\leb\lr{x,y}=\frac{\leb^n}{(2\pi)^n}\int_{|\xi|_{g_y}<1} e^{i\leb\psi(x,y,\xi)}\,\frac{d\xi}{\sqrt{\abs{g_y}}} \, +\,O\lr{\leb^{n-1}},
\end{equation}
where in H\"ormander's terminology, the phase function $\psi$ is adapted to the principal symbol $\abs{\xi}_{g_y}$ of $\sqrt{\Delta_g}.$ After \cite[Theorem 4.4]{Hor}, H\"ormander remarks that the choice of $\psi$ is not unique. However, every adapted phase function satisfies 
\[\psi(x,y,\xi)=\inprod{x-y}{\xi}+O(\abs{x-y}^2\abs{\xi}).\] 
In particular, since $\langle\exp_y^{-1}(x),\xi \rangle_{g_y}=\inprod{x-y}{\xi}+O(\abs{x-y}^2\abs{\xi})$, Taylor expanding \eqref{E:H LWL} yields for any $r_0>0$
\[\sup_{\dist(x,y)<r_0/\leb}\abs{E_\leb\lr{x,y}-\frac{\leb^n}{(2\pi)^n}\int_{|\xi|_{g_y}<1} e^{i\leb\langle \exp_y^{-1}(x),\,\xi\rangle_{g_y}}\,\frac{d\xi}{\sqrt{\abs{g_y}}}}=O(\leb^{n-1}).\]
Changing from one adapted phase to another produces, a priori, an error of $O(\leb^{n-1})$ in \eqref{E:H LWL}. With the additional assumption that $x,y$ are near a non self-focal point, Theorem \ref{close to diagonal} therefore extends H\"ormander's result in two ways. First, our careful choice of phase function $\langle \exp_y^{-1}(x),\,\xi\rangle_{g_y}$ allows us to obtain a $o(\leb^{n-1})$ estimate on $R$ while keeping the amplitude equal to $1$. Second, we allow $\dist(x,y)$ to shrink arbitrarily slowly with $\leb.$

H\"ormander's phase functions $\psi(x,y,\xi)$ are difficult to analyze directly when $x\neq y$ since they are the solutions to certain Hamilton-Jacobi equations (cf \cite[Definition 3.1]{Hor} and \cite[(29.1.7), vol. 4]{Hor 1-4}) which we can not describe explicitly. Instead, in proving Theorem \ref{close to diagonal}, we use a parametrix for the half-wave operator $U(t)=e^{-it\sqrt{\Delta_g}}$ with the geometric phase function $\phi:\R \x M \x T^*M:\gives \R$ given by $\phi(t,x,y,\xi)=\inprod{\exp_y^{-1}(x)}{\xi}-t\abs{\xi}_{g_y}.$ Such a parametrix was previously used by Zelditch in \cite{Zel2}, where a construction for the amplitude was omitted. Our construction, given in \S \ref{S:Parametrix}, makes clear the off-diagonal behavior of $E_\leb(x,y)$ and uses the results of Laptev-Safarov-Vassiliev \cite{LSV}, who treat FIOs with global phase functions. 

Using the phase function $\phi$ simplies our computations considerably since the half-density factor $\sqrt{\det \phi_{x,\xi}(t,x,y,\xi)}$, that comes up in the usual parametrix construction for $U(t)$ acting on half-densities, is independent of $t,\xi.$ This makes it easy to obtain the amplitude in a parametrix for $U(t)$ acting on functions from a that of $U(t)$ acting on half-densities. For more details, see the outline of the proof of Theorem \ref{close to diagonal} given in \S \ref{S: outline main thm} as well as \S \ref{S:Parametrix}, especially \eqref{E:Theta}.


The error estimate in \eqref{E:H LWL} is sharp on Zoll manifolds (see \cite{Zel3}), such as the round sphere. The majority of the prior estimates on $R(x,y,\leb)$ actually treat the case $x=y.$ Notably, B\'erard showed in \cite{Be} that on all compact manifolds of dimension $n\geq 3$ with non-positive sectional curvatures and on all Riemannian surfaces without conjugate points we have $R(x,x,\leb)= O( \leb^{n}/\log \leb).$ The $O(\leb^{n-1})$ error in the Weyl asymptotics for the spectral counting function
\begin{align*}
  \#\set{j:\; \leb_j \in [0,\leb]}~&=\int_M  E_\leb(x,x)dv_g(x)\\
&=\lr{\frac{\leb}{2\pi}}^n\vol_g(M)\cdot \vol_{\R^n}(B_1)+\int_M R(x,x,\leb) dv_g(x),
\end{align*}
has also been improved under various assumptions on the structure of closed geodesics on $(M,g)$ (see \cite{Be, Co, DG, Iv, Nic, PeT, R, SV}). For instance, Duistermaat-Guillemin \cite{DG} and Ivrii \cite{Iv} prove that $\int_M R(x,x,\leb) dv_g(x)=o(\leb^{n-1})$ if $(M,g)$ is aperiodic (i.e the set of all closed geodesics has measure zero in $S^*M$).

Also related to this article are lower bounds for $R(x,y, \leb)$ obtained by Jakobson-Polterovich in \cite{JP} as well as estimates on averages of $R(x,y, \leb)$ with respect to either $y \in M$ or $\leb \in \R_{>0}$ studied by Lapointe-Polterovich-Safarov in \cite{LPS}. 

\subsection{Discussion of Theorem \ref{T:Main}} \label{S: discussion Thm 2}
The scaling asymptotics \eqref{E:Scaling Limit 0} were first stated - without proof and without any assumptions on $\mathcal L_{x_0}$ - by Zelditch in \cite[Theorem 2.1]{Zel1}. When $(M,g)=(S^2,g_{round})$ is the standard $2$-sphere, the square roots of the Laplace eigenvalues are 
$\leb_k=k\cdot \sqrt{1+{1}/{k}}$ for $k\in \Z_+,$
and $\mathcal L_{x_0}=S_{x_0}^*M$ since the geodesic flow is $2\pi$-periodic. There is therefore no $x_0\in S^2$ satisfying the assumptions of Theorem \ref{T:Main}. Nonetheless, Equation \eqref{E:Scaling Limit 2} holds with $E_\leb$ replaced by the kernel of the spectral projection onto the $\leb_k^2$ eigenspace and is known as Mehler-Heine asymptotics (cf \S 8.1 in \cite{Sze}). More generally, on any Zoll manifold, the square roots of Laplace eigenvalues come in clusters that concentrate along an arithmetic progression. The width of the $k^{th}$ cluster is on the order of $k^{-1},$ and we conjecture that the scaling asymptotics \eqref{E:Scaling Limit 2} hold for the spectral projectors onto these clusters (see \cite{Zel3} for background on the spectrum of Zoll manifolds). 

If one perturbs the standard metric on $S^2$ or on a Zoll surface, one can create smooth metrics possessing self-focal points $x_0$ where only a fraction of the measure of initial directions at $x_0$ give geodesics that return to $x_0$. These points complicate the remainder estimate for the general case. Indeed, it was pointed out to the authors by Safarov that even on the diagonal there is a two-term asymptotic formula with the second term of the form $Q(x,\leb)\leb^{n-1},$ where $Q$ is a bounded function.
The function $Q$ is identically zero if $x_0$ is non self-focal or if a full measure of geodesics emanating from $x_0$ return to $x_0$ at the same time. In general, however, $Q$ will contribute an extra term on the order of $\leb^{n-1}$ to the asymptotics in \eqref{E:Scaling Limit 2}. We refer the interested reader to \S 1.8 in \cite{SV}. 



\subsection{Notation}\label{Ss:Notation} Given a Riemannian manifold $(M,g)$ we write $\vol_g(M)$ for its volume, $\dist:M \times M \to \R$ for the induced distance function and $\inj(M,g)$ for its injectivity radius. For $x \in M$ we write $S^*_xM$ for the unit sphere in the co-tangent fiber $T_x^*M.$ We denote by $\langle\cdot , \cdot\rangle_{g_x}: T_x^*M \times T_x^*M \to \R$ the Riemannian inner product on $T_x^*M$ and by $|\cdot|_{g_x}$ the corresponding norm. When $M=\R^n$ we simply write $\langle\cdot , \cdot\rangle$ and $|\cdot|.$ In addition, for $(x,\xi) \in T^*M,$ we will sometimes write $g^{{1}/{2}}_x (\xi)$ for the square root of the matrix $g_x$ applied to the covector $\xi$, and we write $\abs{g_x}$ for the determinant of $g_x.$ 

We denote by $S^k$ the space of classical symbols of degree $k,$ and we will write $S^k_{hom}\subseteq S^k$ for those symbols that are homogeneous of degree $k$. We also denote by $\Psi^k(M)$ the class of pseudodifferential operators of order $k$ on $M.$

\subsection{Acknowledgements} It is our pleasure to thank I. Polterovich, C. Sogge, J. Toth  and particularly Y. Safarov and S. Zelditch for providing detailed comments on earlier drafts of this article. We are also grateful to an anonymous referee whose suggestions and corrections significantly improved the exposition. In particular, the referee pointed us to \cite{LSV}, simplifying our previous parametrix construction in \S 3, and exposed an error in the original version of Proposition \ref{P:Smoothed Estimate}. Fixing our mistake ultimately lead to a more efficient proof. The first author would also like to thank B. Xu for sharing unpublished proofs of some results in \cite{Bin}.

\section{Outline for  the Proof of Theorem \ref{close to diagonal} }\label{S: outline main thm}
Fix $(M,g)$ and a non self-focal point $x_0\in M$. Theorem \ref{close to diagonal} follows from the existence of a constant $c>0$ so that for all $\ep>0$ there exist $\tilde \leb_\ep>0$, an open neighborhood $\mathcal U_\ep$ of $x_0$, and a positive constant $c_{\ep}$,  so that 
\begin{equation}
\sup_{x,y\in \mathcal U_\ep}|R(x,y, \leb)|\leq c\, \ep\,\leb^{n-1}+c_{\ep}\leb^{n-2}\label{E:Main Goal bis}
\end{equation}
for all $\leb \geq \tilde \leb_\ep$.
Indeed, if $r_\leb$ is a positive function with $\lim_{\leb \to \infty} r_\leb=0$, then it suffices to choose $\leb_\ep:=\max\{ \tilde \leb_\ep\,, \,\inf\{\leb:\, B(x_0, r_\leb) \subset \mathcal U_\ep \}\}$ to get  
$$\sup_{x,y\in B(x_0, r_\leb)}|R(x,y, \leb)|\leq c\, \ep\,\leb^{n-1}+c_{\ep}\leb^{n-2}\qquad \forall \leb \geq \leb_\ep.$$

By the definition \eqref{E:LWL} of $R$ and the definition \eqref{E:Big Ker Def} of $E_\leb$, we seek to find a constant $c>0$ so that for all $\ep>0$ there exist $\tilde \leb_\ep>0$, an open neighborhood $\mathcal U_\ep$ of $x_0$, and a positive constant $c_{\ep}$ satisfying 
\begin{equation}
\sup_{x,y\in \mathcal U_\ep}\abs{E_\leb(x,y)-\frac{\leb^n}{(2\pi)^n}\int_{\abs{\xi}_{g_y}<1}e^{i\leb\langle\exp_y^{-1}(x) ,\xi\rangle_{g_y}} \frac{d\xi}{\sqrt{|g_y|}}} \leq c\, \ep\,\leb^{n-1}+c_{\ep}\leb^{n-2}\label{E:Main Goal}
\end{equation}
for all $\leb \geq \tilde \leb_\ep$. We prove \eqref{E:Main Goal} using the so-called wave kernel method. That is, we use that the derivative of the spectral function is the inverse Fourier transform of the fundamental solution of the half-wave equation on $(M,g):$
\begin{equation}\label{E:WKM}
 E_\leb(x,y)=\int_0^\leb \sum_j \delta(\mu-\leb_j)\varphi_j(x)\varphi_j(y)~d\mu =\int_0^\leb \mathcal F_{t\gives \mu}^{-1}(U(t,x,y))(\mu)\, d\mu,
\end{equation}
where $\mathcal F^{-1}$ denotes the inverse Fourier transform and $U(t,x,y)$ is the Schwartz kernel of $e^{-it\sqrt{\Delta_g}}.$ The singularities of $U(t,x,y)$ control the $\leb\gives \infty$ behavior of $E_\leb.$ We first study the contribution of the singularity of $U(t,x,y)$ coming at $t=\dist(x,y)$ by taking a Schwartz function $\rho \in \mathcal S(\R)$ that satisfies $\supp \lr{\hat{\rho}}\subseteq \lr{-\inj(M,g),\inj(M,g)}$ and 
\begin{equation}\label{E:Rho Def}
 \hat{\rho}(t)= 1 \qquad \text{for all}\qquad  \abs{t}<\tfrac{1}{2}{\,\inj(M,g)}.
 \end{equation}
We prove in \S \ref{S Est 1} the following proposition, which shows that \eqref{E:Main Goal} holds with $E_\leb$ replaced by $\rho*E_\leb.$
\begin{proposition}[Smoothed Projector]\label{P:Smoothed Est 1}
Let $(M,g)$ be a compact, smooth, Riemannian manifold of dimension $n \geq 2$, with no boundary. Then, there exist constants $c,C>0$ so that 
\begin{equation}
\Big|\rho*E_\leb(x,y)-\frac{1}{(2\pi)^n}\int_{\abs{\xi}_{g_y}<\leb}e^{i\langle\exp_y^{-1}(x),\xi\rangle_{g_y}}\frac{d\xi}{\sqrt{|g_y|}} \Big|
\leq   c \dist(x,y) \leb^{n-1} +C \leb^{n-2}  \label{E:Main Goal Smooth 1}
\end{equation}
for all $x,y \in M$ with $\dist(x,y)\leq\tfrac{1}{2}\inj(M,g)$ and all $\leb>0.$
\end{proposition}
Note that Proposition \ref{P:Smoothed Est 1} does not assume that $x,y$ are near a non self-focal point. The reason is that convolving $E_\leb$ with $\rho$ multiplies the half-wave kernel $U(t,x,y)$ in \eqref{E:WKM} by the Fourier transform $\hat{\rho}(t),$ which cuts out all but the singularity at $t=\dist(x,y).$ The proof of \eqref{E:Main Goal Smooth 1} relies on the construction in \S \ref{S:Parametrix} of a short time parametrix for $U(t),$ which differs from the celebrated H\"ormander parametrix becaues it uses the coordinate-independent phase function 
\begin{equation}\label{E: phi}
\phi(t,x,y,\xi):=\langle {\exp_y^{-1}(x)}, \xi \rangle_{g_y}-t\abs{\xi}_{g_y}\qquad (t,x,y,\xi)\in \R \x M \x T^*M. 
\end{equation} 


It remains to estimate the difference $|E_\leb(x,y)- \rho*E_\leb(x,y)|,$ which is the content of the following result.

\begin{proposition}[Smooth vs rough Projector]\label{P:Smoothed Est 2} 
Let $(M,g)$ be a compact, smooth,  Riemannian manifold of dimension $n \geq 2$, with no boundary.  Let $x_0\in M$ be a non self-focal  point. Then, there exists $c>0$ so that for all   $\ep>0$  there exist  an open neighborhood $\mathcal U_\ep$ of $x_0$ and a  positive constant $c_{\ep}$ with
\begin{equation}
\sup_{x,y\in \mathcal U_\ep}\abs{E_\leb(x,y)-\rho*E_\leb(x,y)}\leq c\, \ep \leb^{n-1}+c_{\ep}\leb^{n-2}\label{E:Main Goal Smooth 2}
\end{equation}
 for all $\leb\geq 1$.
\end{proposition}

The assumption that $x,y$ are near a non self-focal point $x_0$ guarantees that the dominant contribution to $E_\leb(x,y)$ comes from the singularity of $U(t,x,y)$ at $t=\dist(x,y)$. Following the technique in \cite{SZ}, we prove Proposition \ref{P:Smoothed Est 2} in \S\ref{S Est 2} by microlocalizing $U(t)$ near $x_0$ (see \S \ref{S:Microlocal Identity}) and applying two Tauberian-type theorems (presented in \S \ref{S:Tauberian}). Relation \eqref{E:Main Goal},  and consequently Theorem \ref{close to diagonal},  are a direct consequence of combining Proposition \ref{P:Smoothed Est 1}  with Proposition \ref{P:Smoothed Est 2}.

\section{Parametrix for the Half-Wave Group}\label{S:Parametrix}
The half-wave group is the one parameter family of unitary operators $U(t)=e^{-it\sqrt{\Delta_g}}$ acting on $L^2(M,g).$ It solves the initial value problem
$$\Big(\frac{1}{i}\partial_t  + \sqrt{\Delta_g}\Big) U(t) =0, \qquad U(0)=Id,$$
and its Schwartz kernel $U(t,x,y)$ is related to the kernel of the spectral projector $E_\leb(x,y)$ via \eqref{E:WKM}. It is well-known (cf \cite{DG, Hor 1-4}) that $U$ is a Fourier integral operator in $I^{-1/4}(\R \x M, M ; \Gamma)$ associated to the canonical relation 
\begin{equation}
 \Gamma= \Big\{ (t, \tau,x, \eta, y,\xi)\in T^*(\R \x M \x M)|\;\; \tau=-|\xi|_{g_y}, \quad G^{t}(y,\xi)=(x,\eta)\Big\},\label{E:Canonical Relation}
\end{equation}
where $G^t$ denotes geodesic flow. 

Our goal in this section is to construct a short time parametrix for $U(t)$ that is similar to H\"ormander's parametrix (cf \cite{Hor}, \cite[\S 29]{Hor 1-4}) but uses the coordinate independent phase function $\phi:\R \times M \times T^*M \to \R$ defined in \eqref{E: phi}. Such a parametrix was used by Zelditch in \cite{Zel2}, where a detailed construction was omitted.  To construct the amplitude we follow Laptev-Safarov-Vassiliev \cite{LSV} who give a detailed treatment of FIOs that are built using global phase functions such as $\phi$. Denote by $\chi \in C^\infty([0, +\infty), [0,1])$ a compactly supported smooth cut-off function with
$$\text{supp} \,\chi \subset [0, \inj(M,g)) \qquad \quad\text{and} \qquad \quad  \chi (s)=1   \;\;\; \text{for}\;\;\; s \in [0, \inj (M,g)/2).$$ 
Further, following \cite{Ber} and \cite[Proposition C.III.2]{BGM}, define
\begin{equation}
\Theta(x,y):=|\text{det}_g D_{\exp_x^{-1}(y)}\exp_x|.\label{E:Theta Def}
\end{equation}
The subscript $g$ means that we use the inner products on $T_{\exp_x^{-1}(y)}(T_xM)$ and $T_y^*M$ induced from $g,$ and as explained in \cite{BGM}, $\Theta(x,y)=\sqrt{|g_x|}$ in normal coordinates at $y.$ The main result of this section is the following. 

\begin{proposition}\label{P:Parametrix}
For $\abs{t}<\inj(M,g)$  we have
\begin{equation}\label{E:Parametrix}
U(t,x,y)= \frac{\chi(\dist(x,y))}{(2\pi)^n\Theta(x,y)^{\frac{1}{2}}} \int_{T_y^*M}e^{i\phi(t,x,y,\xi)}A(t,y,\xi) \frac{d\xi}{\sqrt{|g_y|}},
\end{equation}
where the equality is modulo smoothing kernels. The amplitude $A$, which is an order $0$ polyhomogeneous symbol, is uniquely determined by $\phi$ modulo $S^{-\infty}$ and satisfies: 
\begin{itemize}
\item For all $y \in M$ and $\xi \in T_y^*M$, 
\begin{equation}\label{E: A on diagonal} 
A(0,y,\xi)=1.
\end{equation}    

\item For $\abs{t}<\inj(M,g)$ and all $(y,\xi)\in T_y^*M,$ we have 
\begin{equation}\label{E: A off diagonal}
A(t,y,\xi)-1\in S^{-1}.
\end{equation}
\end{itemize}
\end{proposition}

There are many choices of amplitude functions in \eqref{E:Parametrix} that depend on $t,x,y,\xi.$ When we write that $A$ is uniquely determined modulo $S^{-\infty}, $ we mean that it is unique among amplitudes that are independent of $x.$ The proof of Proposition \ref{P:Parametrix} is divided into two steps. First, we prove in \S \ref{S:Phase} that $\phi$ parametrizes $\Gamma$. Then, in \S \ref{S:Amplitude} we construct the amplitude $A$.

\subsection{Properties of the phase function}\label{S:Phase}
Throughout this section, we will denote by $\mathcal T_{y\gives x}:T_y^*M\to T_x^*M$ the parallel transport operator (along the unique shortest geodesic from $x$ to $y$) for all $x$ and $y$ sufficiently close. We will use that
\begin{equation}
\mathcal T_{y\gives x} \exp_y^{-1}(x)=-\exp_x^{-1}(y)\qquad \text{and}\qquad \mathcal T_{y\gives x}=\mathcal T_{x\gives y}^*.\label{E:Parallel Transport}
\end{equation}
\begin{Lem}\label{L: eikonal eqn}
The phase function $\phi(t,x,y,\xi)$ parametrizes the canonical relation $\Gamma$ for $\abs{t}<\inj(M,g)$ and $\dist(x,y)<\inj(M,g)/2$ in the
sense that 
\begin{equation}
\Gamma=i_\phi(C_\phi)\label{E:Image}
\end{equation}
is the image of the critical set
$$C_\phi=\Big\{(t,x,y,\xi) \in \R \times M \times T^*M|\;\; x=\exp_y \Big(\frac{t \xi}{|\xi|_{g_y}} \Big) \Big\}$$
under the immersion $i_\phi(t,x,y,\xi)=(t, d_t \phi, x, d_x \phi, y,-d_y \phi).$
\end{Lem}
\begin{proof}
When $\abs{t}<\inj(M,g),$ we have that $(t,x,y,\xi)\in C_\phi$ if and only if $t=0$ and $x=y$, or 
 \[t=\dist(x,y)\neq 0\qquad \text{and}\qquad \xi/\abs{\xi}_{g_y} = \exp_y^{-1}(x)/\dist(x,y).\] 
To prove \eqref{E:Image} when $t=0,$ we must show that 
\begin{equation}
i_\phi(0,x,x,\xi)=\set{(0,-\abs{\xi}_{g_x}, x,\xi, x, \xi),~~ \xi \in T_x^*M}=\Gamma|_{t=0}.\label{E:Param Goal 1}
\end{equation}
Since $d_x|_{x=y}\exp_y^{-1}(x)$ is the identity on $T_y^*M,$
\[d_x|_{x=y}\phi(0,x,y,\xi)=\xi.\]
Next, using \eqref{E:Parallel Transport}, we have 
\[\phi(0,x,y,\xi)=\inprod{-\exp_x^{-1}(y)}{\mathcal T_{y\gives x} \xi}_{g_x}.\]
Therefore, 
\[d_y|_{y=x}\phi(0,x,y,\xi)=-\xi,\]
which proves \eqref{E:Param Goal 1}. To establish \eqref{E:Image} when $t\neq 0,$ we write  
\begin{align}\label{E:dphi}
 \dell_{x_k}\phi(t,x,y,\xi) = \sum_{i,j} g^{ij}(y) \dell_{x_k} \left[\exp_{y}^{-1}(x)\right]_i \xi_j\, ,\qquad k=1,\ldots, n.
\end{align}
Since $d_x \dist(x,y)= -\exp_x^{-1}(y)/\dist(x,y)$, evaluating \eqref{E:dphi} at \[\xi=\abs{\xi}_{g_y}\exp_y^{-1}(x)/\dist(x,y),\] we obtain
\begin{equation}
d_x\phi(t,x,y,\xi)=\frac{\abs{\xi}_{g_y}}{2 \dist(x,y)} d_x \left[\dist(x,y)^2\right] = \abs{\xi}_{g_y}d_x \dist(x,y)=-\abs{\xi}_{g_y}\frac{\exp_x^{-1}(y)}{\dist(x,y)}.\label{E:Dx phi}
\end{equation}
Since $G^t(y,\exp_y^{-1}(x))=(x,-\exp_x^{-1}(y)),$ it remains to check that 
\[-d_y\phi(t,x,y,\xi)=\abs{\xi}_{g_y}\frac{\exp_y^{-1}(x)}{\dist(x,y)},\]
which we verify in normal coordinate at $y.$ We have that 
\[d_z|_{z=y}\abs{\xi}_z=0\qquad \text{and}\qquad \dell_{z_k}|_{z=y} \lr{\exp_z^{-1}(x)}_j=-\delta_{kj}.\] 
Thus, 
\begin{equation*}
\dell_{z_k}|_{z=y}\phi(t,x,z,\xi)=-  \xi_k.
\end{equation*}
Evaluating at $\xi=\abs{\xi} \cdot x/\abs{x},$ we find that
\begin{equation*}
-d_y\phi(t,x,y,\xi)=\abs{\xi}\cdot \frac{x}{\abs{x}}=\abs{\xi}_{g_y}\frac{\exp_y^{-1}(x)}{\dist(x,y)},
\end{equation*}
as desired. 
\end{proof}

We need one more Lemma before constructing the amplitude $A$ in Proposition \ref{P:Parametrix}.

\begin{Lem}\label{L:Delta Ker}
Let $\beta:M\x M \gives \R$ be any smooth function such that $\beta(x,x)=1.$ The kernel of the identity operator acting on functions relative to the Riemannian volume form $\sqrt{\abs{g_y}}dy$ admits the following representation as an oscillatory integral:
\begin{align}
\delta(x,y)&=\frac{\chi(\dist(x,y)) }{(2\pi)^n}\beta(x,y)\int_{T_x^*M} e^{-i\langle\exp_x^{-1}(y),\eta \rangle_{g_x}}\frac{d\eta}{\sqrt{|g_x|}}\label{E:Delta Ker}\\
\notag &=\frac{\chi(\dist(x,y))}{(2\pi)^n} \beta(x,y)\int_{T_y^*M} e^{i\langle\exp_y^{-1}(x), \xi\rangle_{g_y}}\frac{d\xi}{\sqrt{|g_y|}}.
\end{align}
\end{Lem}
\begin{proof}
Fix $x\in M$ and let $f\in C^\infty(M).$ Without loss of generality, assume that $f$ is supported in an open set $U \subset B(x, \inj(M,g))$ that contains the point $x$. Set $V= \exp_x^{-1}(U) \subset \R^n$ and consider normal coordinates at $x$:
\begin{equation}\label{E: normal coord}
h:V \to U, \qquad h(z)= \exp_x(z).
\end{equation}
The pairing of the RHS  of \eqref{E:Delta Ker}  with $f$ is then
\begin{align*}
\frac{1}{(2\pi)^n}\int_{\R^n}\int_{\R^n} e^{-i\inprod{z}{\eta}}\chi(|z|)f(h(z))\beta(0,z)\sqrt{|g_{h(z)}|} dz d\eta
&=\chi(|0|)\,f(h(0))\,\sqrt{|g_{h(0)}|}\beta(0,0)\\
&=f(x).
\end{align*}
This proves \eqref{E:Delta Ker}. To explain why the two oscillatory integrals in the statement of the present Lemma define the same distribution, we will use the parallel transport operator (see \eqref{E:Parallel Transport}). We write \eqref{E:Delta Ker} as
\begin{equation}
\frac{\chi(\dist(x,y))}{(2\pi)^n} \beta(x,y)\int_{T_x^*M} e^{i\inprod{\exp_y^{-1}(x)}{\mathcal T_{y\gives x}\eta}_{g_y}}\frac{d\eta}{\sqrt{|g_x|}}\label{E:Delta Ker 1}
\end{equation}
Let $(y^1,\ldots, y^n)$ be any local coordinates near $x$. We note that for every $y,$ the collection of covectors $\set{g_y^{1/2}dy^j|_y}_{j=1}^n$ is an orthonormal basis for $T_y^*M.$ Hence, the Lebesgue measure on $T_y^*M$ in our coordinates is $\abs{g_y}^{1/2}dy^1|_y\wedge\cdots\wedge dy^n|_y,$ and since $\mathcal T_{y\gives x}$ is an isometry, 
\[\text{if}\quad \xi = \mathcal T_{y\gives x}\eta, \qquad \text{then} \qquad d\xi =\frac{\abs{g_y}^{1/2}}{\abs{g_x}^{1/2}} d\eta.\]
This allows us to change variables in \eqref{E:Delta Ker 1} to obtain the integral over $T_y^*M$ in the statement of the Lemma. 
\end{proof}

\subsection{Construction of the amplitude}\label{S:Amplitude}
To construct the amplitude $A$ in Proposition \ref{P:Parametrix}, let us write $\twiddle{U}(t)$ for the wave operator acting on sections of the half-density bundle $\W^{1/2}(M).$ Lemma \ref{L: eikonal eqn} combined with Theorem 3.4 in \cite{LSV} (or Proposition 25.1.5 in \cite{Hor 1-4}) shows that there exists a polyhomogeneous symbol ${A}$ of order $0$ that is supported in a neighborhood of $C_\phi$ for which
\begin{equation}\label{E: tilde U}
\twiddle{U}(t,x,y)=\frac{\chi(\dist(x,y))}{(2\pi)^n} \int_{T_y^*M}e^{i\phi(t,x,y,\xi)}{A}(t,y,\xi)d_\phi(t,x,y,\xi)d\xi \quad (\text{mod } C^\infty),
\end{equation}
where 
\[d_\phi=\sqrt{|\det d_{x,\xi}\phi|}\in \W_x^{1/2}(M)\otimes \W_y^{-1/2}(M)\] 
is a $\tfrac{1}{2}$-density in $x$ and a $(-\tfrac{1}{2})$-density in $y.$ Since $d\xi$ behaves like a $1$-density in $y$, $\twiddle{U}(t,x,y) \in \W_x^{1/2}(M)\otimes \W_y^{1/2}(M).$ The square root of the Riemannian volume form
\[g_y^{1/4}=\abs{g_y}^{1/4}\abs{dy}^{1/2}\in \W_y^{1/2}(M)\]
identifies $L^2$ global sections $\Gamma\lr{\W^{1/2}(M)}$ with $L^2(M)$ via
\[f(y)\in L^2(M)\quad \mapsto \quad f(y)\cdot g_y^{1/4}\in \Gamma\lr{\W^{1/2}(M)}.\]
Then, computing in normal coordinates at $y,$ we have 
\begin{equation}
d_\phi(t,x,y,\xi){g_y^{1/4}}{g_x^{-1/4}}=\frac{1}{|g_x|^{\frac{1}{4}}}=\frac{1}{\Theta(x,y)^{\frac{1}{2}}}.\label{E:Theta}
\end{equation}
In addition, since $U(t,x,y)=\twiddle{U}(t,x,y) {g_x^{-1/4}g_y^{-1/4}}$,
relation \eqref{E:Theta} gives
\begin{align}\label{E: no tilde U}
U(t,x,y)
=\frac{\chi(\dist(x,y))}{(2\pi)^n \Theta(x,y)^{\frac{1}{2}}}\int_{T_y^*M}e^{i\phi(t,x,y,\xi)}{A}(t,y,\xi)\frac{d\xi}{\sqrt{\abs{g_y}}}  \qquad (\text{mod }C^\infty).
\end{align}
Write ${A}\sim \sum_{j\geq 0} {A}_{-j}$ for the polyhomogeneous expansion of ${A}.$ Note that
 \[{A}_0(t,y,\xi)=1 \qquad \text{for all}\; t\]
because the principal symbol $\twiddle{U}(t)$ is independent of $t$ and equals $1$ at $t=0$ \cite[Theorem 4.1]{LSV}. Next, since 
\[\twiddle{U}(0,x,y) =\frac{\chi(\dist(x,y))}{(2\pi)^n }\int_{T_y^*M}e^{i\phi(t,x,y,\xi)}A(0,y,\xi)d_\phi(t,x,y,\xi)\frac{d\xi}{\sqrt{\abs{g_y}}}\]
is a kernel for the identity modulo $C^\infty$ and $A(0,y,\xi)$ is uniquely determined by $\phi$ mod $S^{-\infty}$ (Theorem 3.4 in \cite{LSV}), it follows from Lemma \ref{L:Delta Ker} and \eqref{E:Theta}, with $\beta(x,y)= \Theta(x,y)^{-\frac{1}{2}}$, that
\[A_{-j}(0,y,\xi)=0 \qquad \text{for all}\; j\geq 1,\]
as desired.

 \section{Microlocalizing the identity operator at non self-focal points }\label{S:Microlocal Identity}
In this section we microlocalize the identity operator near a non self-focal point $x_0$. For every $\ep>0$ we make a microlocal decomposition of the identity $Id=B_\ep + C_\ep$ near $x_0$, where the operator $B_\ep$ is supported on the set of ``bad" loopset directions and it is built so that its support has measure smaller than $\ep$. This construction follows closely that of Sogge-Zelditch in \cite{SZ}.
\begin{lemma}\label{P:Microlocal Identity}
There exists a constant $\gamma>0$ so that for every $\ep>0$ there is a neighborhood $\mathcal O_\ep$ of $x_0$, a function $\psi_\ep \in C^\infty_c(M)$ and real valued operators $B_\ep,C_\ep\in \Psi^0(M)$ supported in $\mathcal \mathcal \mathcal O_\ep$ satisfying the following properties:
\begin{enumerate}
\item For every $\ep,$ $\supp(\psi_\ep) \subset \mathcal O_\ep$ and $\psi_\ep=1$ on a neighborhood of $x_0$.
\item For every $\ep,$
\begin{equation}
B_\ep+C_\ep=\psi_\ep^2.\label{E:Microlocal Identity}
\end{equation}
\item $U(t)C_\ep^*$ is a smoothing operator for $\;\tfrac{1}{2}{\inj(M,g)}<\abs{t}<\frac{1}{\ep}.$
\item Denote by $b_0$ and $c_0$ the principal symbols of $B_\ep$ and $C_\ep$ respectively. Then, for all $x \in M$ we have 
\begin{equation}
  \label{E:Symbol Est 1}
 \frac{1}{\ep}\int_{\abs{\xi}_{g_x}\leq 1} \abs{b_0(x,\xi)}^2d\xi + \int_{\abs{\xi}_{g_x}\leq 1} \abs{c_0(x,\xi)}^2 d\xi \leq \gamma,
\end{equation}
and both $b_0$ and $c_0$ are constant in an open neighborhood of $x_0$.
\end{enumerate}
\end{lemma}
\begin{proof}
For every $x,y\in M$ and $\xi \in S_x^* M$ define the loopset function
\[\mathcal L^*(x,y,\xi)=\inf \big\{t> 0\;|\;\, \exp_x (t\xi) = y \big\}\]
with $\mathcal L^*(x,y,\xi)=+\infty$ in case the infimum is taken over the empty set. Unlike the loopset function studied in \cite{SZ}, we are interested in $x\neq y$ (but with $\dist(x,y)<\tfrac{1}{2}{\inj(M,g)}).$ 

Fix a coordinate chart $(\kappa_{x_0},\mathcal V_{x_0})$ containing $x_0$ with $\kappa_{x_0}:\mathcal V_{x_0} \subset \R^n \to M$.
We first note that the function $f:\mathcal V_{x_0}\times \mathcal V_{x_0}\times S^{n-1} \to \mathbb R$ defined as $f(x,y,\xi)=1/  \mathcal L^*(x,y,\xi)$ is upper semicontinuous and so by the proof of \cite[Lemma 3.1]{SZ} there exist a neighborhood $\mathcal N_\ep \subset \mathcal V_{x_0}$ of $x_0$ and an open set $\Omega_\ep \subset S^{n-1}$ for which 
\begin{align}
 \mathcal L^*(x,y,\xi)&>\frac{1}{\ep}\qquad \text{in}\;\;\;  \mathcal N_\ep \times \mathcal N_\ep \times \Omega_\ep^c, \label{E:LL Control 1} \\
 \abs{\W_{\ep}} &\leq \ep. \label{E:LL Control 2}
\end{align}
In addition, there exists a function $\varrho_\ep \in C^\infty(S^{n-1}, [0,1])$ satisfying that  $\varrho_\ep\equiv 1$ on $\Omega_\ep$, $\varrho_\ep (\xi)=\varrho(-\xi)$ for all $\xi \in S^{n-1}$, and $|\supp(\varrho_\ep)|<2\ep$. In particular, 
$$ \mathcal L^*(x,y,\xi)>\frac{1}{\ep}\qquad  \text{on}\qquad \mathcal N_\ep \times \mathcal N_\ep \times \supp(1-\varrho_\ep).$$

As in \cite{SZ} we choose a real-valued function $\tilde \psi_\ep \in C^\infty_c(\mathbb R^n)$ with $\supp(\tilde \psi_\ep) \subset \mathcal N_\ep$ and equal to $1$ in a neighborhood of $\kappa_{x_0}^{-1}(x_0)$. Define symbols on $\R^{3n}$ by 
$$\tilde b_\ep(x,y,\xi)= \tilde \psi_\ep(x) \tilde \psi_\ep(y) \varrho_\ep \lr{\tfrac{\xi}{\abs{\xi}}} \qquad \text{and} \qquad \tilde c_\ep(x,y,\xi)= \tilde \psi_\ep(x) \tilde \psi_\ep(y)\Big(1-\varrho_\ep \lr{\tfrac{\xi}{\abs{\xi}}}\Big),$$
and consider their respective quantizations $Op(\tilde b_\ep), Op(\tilde c_\ep) \in \Psi^0(\mathbb R^n)$. Properties (1) and (2) follow from setting 
$$B_\ep:=(\kappa_{x_0}^{-1})^* \,Op(\tilde b_\ep), \quad C_\ep:=(\kappa_{x_0}^{-1})^* \,Op(\tilde c_\ep)$$
and
$$ \mathcal O_\ep=\kappa_{x_0}(\mathcal N_\ep), \quad\quad   \psi_\ep:= (\kappa_{x_0}^{-1})^* \tilde \psi_\ep.$$
Note that if for some time $\tfrac{1}{2}\inj(M,g)<t< \tfrac{1}{\ep}$ we have $\exp_{x}(t \tfrac{\xi}{|\xi|})=y$ for some $x,y \in M$ and $\xi \in T_x^*M$, then $\mathcal L^*(x,y, \tfrac{\xi}{|\xi|}) \leq \frac{1}{\ep}$, and the latter implies $\tilde c_\ep(x,y, {\xi})=0$.  Therefore, we see that if we write $c_\ep$ for the symbol of $C_\ep$, then
\[c_\ep(x,y,\xi)=0 \quad \text{if}\quad \lr{t,x,y;\tau, \xi, \eta}\in \Gamma  \quad \text{with} \quad \tfrac{1}{2}\inj(M,g)<t< \tfrac{1}{\ep},\]
where $\Gamma$ is the canonical relation underlying $U(t)$ (see \ref{E:Canonical Relation}). Thus, the kernel of $U(t)C_\ep^*$ is a smooth function for $\tfrac{1}{2}\inj(M,g)<t< \tfrac{1}{\ep}$ and for $(x,y)$ in $\mathcal O_\ep \times \mathcal O_\ep$ which is precisely statement (3). For all $x\in \mathcal N_\ep$ we have that the principal symbols $b_0,c_0$ satisfy the inequality \eqref{E:Symbol Est 1} since $|\supp \varrho_\ep| < 2\ep$. Also, since $b_\ep$ and $c_\ep$ are real valued and invariant under $\xi \mapsto -\xi$, we have that $B_\ep$ and $C_\ep$ are real valued as well.
\end{proof}

\begin{remark}\label{R: -1 term}
By construction, the subprincipal symbols of $B_\ep$ and $C_\ep$ (acting on half-densities) are zero in a neighborhood of $x_0$. Indeed, the principal symbols are constant as functions of $x$ in a neighborhood of $x_0,$ and in the coordinates $\kappa_{x_0}$ used in Lemma \ref{P:Microlocal Identity} the total symbols of $B_\ep$ and $C_\ep$ are homogeneous functions of order zero. Thus, in any coordinates, the order $-1$ parts of the polyhomogeneous expansions of the total symbols of $B_\ep$ and $C_\ep$ vanish in an neighborhood of $x_0.$
\end{remark}

\begin{remark}\label{remark R_Q} We record precise asymptotics for the on-diagonal behavior of $QEQ^*(x,x, \mu)$ for all $x\in \mathcal O_\ep$ and $Q \in \{Id, B_\ep, C_\ep\}$. Write $q_0$ for the principal symbol of $Q.$ Using that the sub-principal symbols of both $Q$ and $QQ^*$ (acting on half-densities) vanish identically in a neighborhood $\twiddle{\mathcal O}_\ep$ of $x_0,$ Lemma 3.2 and Lemma 3.3 in \cite{SZ} show that there exist constants $c, c_\ep>0$ so that for all $x \in \twiddle{\mathcal O}_\ep$
\begin{align*}
QEQ^*(x,x, \leb)
=& \frac{1}{(2\pi)^n} \int_{|\xi|_{g_x}<\leb} |q_0(x,\xi)|^2\; d\xi +R_Q(x,x,\leb)
\end{align*}
with
\begin{equation}\label{R_Q}
\abs{R_Q(x,x,\leb)}\leq c \,{\ep}\leb^{n-1}+c_\ep \leb^{n-2}
\end{equation}
for all $\leb \geq 1$. 
 We note that a similar result is obtained in \cite[Theorem 1.8.7]{SV} with the difference  that the latter is proved for points $x$ that are non-focal. 
\end{remark}

\section{Smoothed projector: proof of Proposition \ref{P:Smoothed Est 1}}\label{S:Smooth Microlocal Proj}
Proposition \ref{P:Smoothed Estimate} below is our main technical estimate on $E_\leb(x,y).$ We use Proposition \ref{P:Smoothed Estimate} to prove  Propositions \ref{P:Smoothed Est 1} and \ref{P:Smoothed Est 2} in \S \ref{S Est 1} and \S \ref{S Est 2} respectively. 

\begin{proposition}\label{P:Smoothed Estimate}
Let $(M,g)$ be a compact, smooth, Riemannian manifold of dimension $n\geq 2$, with no boundary. Let $\ep>0$ and $Q \in \{Id, B_\ep, C_\ep\}$ for $B_\ep$ and $C_\ep$ as introduced in Lemma \ref{P:Microlocal Identity}. Let $q_0$ be the principal symbol of $Q$. Then, for all $x,y \in \mathcal \mathcal O_\ep$ with $\dist(x,y)\leq \tfrac{1}{2}\inj(M,g)$ and all $\mu \geq 1$, we have 
 \begin{align}
\partial_\mu(\rho*EQ^*)(x,y,\mu)
\notag &=\frac{\mu^{n-1}}{(2\pi)^n\Theta(x,y)^{\frac{1}{2}}}\left[\int_{S_y^*M}e^{i \langle \exp_y^{-1}(x),\w \rangle_{g_y}} q_0(y,\w) \frac{d\omega}{\sqrt{|g_y|}}\right.\\
&\left.\;\;\; +\int_{S_y^*M}e^{i \langle \exp_y^{-1}(x),\w \rangle_{g_y}} D_{-1}^Q(y,\w)\frac{d\omega}{\sqrt{|g_y|}}\right]+ W(x,y,\mu). \label{E:Smoothed Est 1}
\end{align}  
Here, $d\w$ is the Euclidean surface measure on $S_y^*M$ and the function $\Theta$ is defined in \eqref{E:Theta Def}. The function  $D_{-1}^Q\in S^{-1}$ and there exists $C>0$ so that for every $\ep>0$
\begin{equation}\label{E:K rln}
D_{-1}^{B_\ep}(y,\xi)+D_{-1}^{C_\ep}(y,\xi)=0 \qquad \forall \,y\in \mathcal O_\ep,
\end{equation}
\begin{equation}\label{E:K rln}
\sup_{x,y\in \mathcal O_\ep}\abs{\int_{S_y^*M}e^{i \langle \exp_y^{-1}(x),\w \rangle_{g_y}} {D_{-1}^Q(y,\w)} \frac{d\omega}{\sqrt{|g_y|}}}\leq C\, \ep.
\end{equation}
In addition, $W$ is a smooth function in $(x,y)$ for which there exists $C>0$ such that for all $\mu >0$
\begin{equation}\label{E:W long range}
\sup_{\dist(x,y)\leq \frac{1}{2} \inj(M,g)}|W(x,y,\mu)| \leq C \lr{ \mu^{n-2}\dist(x,y)+ (1+\mu)^{n-3}}.
\end{equation}
\end{proposition}

\begin{proof}
Let $x,y \in M$ with $\dist(x,y)\leq \tfrac{1}{2}\inj(M,g)$. Note that 
\begin{equation}\label{E: convolution}
\partial_\mu (\rho * EQ^*)(x,y,\mu)= \frac{1}{2\pi} \int_{-\infty}^{+\infty} e^{it \lambda} \hat \rho (t) U(t) Q^*(x,y) dt.
\end{equation}
We start by rewriting $U(t) Q^*(x,y) $ using the parametrix \eqref{E:Parametrix} for $U(t).$ We have
\begin{equation}\label{E: UQ}
U(t) Q^*(x,y)=\frac{\chi(d_g(x,y))}{\lr{2\pi}^n\Theta(x,y)^{\frac{1}{2}}} \int_{T_y^*M}  e^{i \langle \exp_y^{-1}(x),\xi \rangle_{g_y}-t|\xi|_{g_y}}  \, D^Q(t,y,\xi)  \frac{d\xi}{\sqrt{|g_y|}},
\end{equation}
for some $D^Q\in S^0$ with polyhomogeneous expansion $D^Q \sim \sum_{j\geq 0}D_{-j}^Q.$ We claim that
\begin{equation}
D_0^Q(0,y,\xi)=q_0(y,\xi),\label{E:D0}
\end{equation}
and that for all $\ep>0$
\begin{equation}
D_{-1}^{B^\ep}(0,y,\xi)+D_{-1}^{C^\ep}(0,y,\xi)=0, \label{E:D1}
\end{equation}
\begin{equation}
\sup_{x,y\in \mathcal O_\ep}\abs{\int_{S_y^*M}e^{i \langle \exp_y^{-1}(x),\w \rangle_{g_y}} {D_{-1}^Q(0,y,\xi)} \frac{d\omega}{\sqrt{|g_y|}}}\leq C\, \ep, \label{E:D2}
\end{equation}
where $C$ is a constant independent of $\ep$.
Indeed, let $\twiddle{U}(t) \twiddle{Q}^*$ denote the operator $U(t) Q^*$ when regarded as acting on half-densities, and note that by the same computations that deduce \eqref{E: no tilde U} from \eqref{E: tilde U} we have
\[\twiddle{U}(t) \twiddle{Q}^*(x,y)=\frac{\chi(d_g(x,y))}{\lr{2\pi}^n} \int_{T_y^*M}  e^{i \langle \exp_y^{-1}(x),\xi \rangle_{g_y}-t|\xi|_{g_y}}  \, {D^Q}(t,y,\xi)  d_\phi(t,x,y,\xi){d\xi}.\]
Since the principal symbols of both $\twiddle{U}$ and $\twiddle{Q}$ are independent of $t$, and $\tilde U(0)=Id$, we know 
\[D^Q_0(t,y,\xi)=q_0(y,\xi).\]
Moreover, note that $D_{-1}^{Id}(0,y,\xi)=0$ by Proposition \ref{P:Parametrix} and  that $D^{Id}$ is uniquely determined modulo $S^{-\infty}$ by the phase function $\phi$ (cf \cite{LSV}). This proves \eqref{E:D1} since on $\mathcal O_\ep$ we have $Id=B_\ep+C_\ep.$  Finally, by the construction of $B_\ep,$ we see that the size of the support of $D_{-1}^{B_\ep}(0,y,\xi)$ is smaller than  a constant times $\ep.$ This proves \eqref{E:D2} for $Q=B_\ep$ and hence for $Q=C_\ep$ since $D_{-1}^{B_\ep}=-D_{-1}^{C_\ep}.$ 

Combining \eqref{E: convolution} and \eqref{E: UQ}, and changing coordinates $\xi \mapsto \mu r \w$ where $(r, \w) \in [0, +\infty) \times S_y^*M$, we obtain that up to an $O(\mu^{-\infty})$ error that
\begin{align}\label{E:Compact Int}
&\Theta(x,y)^{\frac{1}{2}} \cdot \partial_\mu (\rho * EQ^*)(x,y,\mu)= \notag\\
&\frac{\mu^n}{\lr{2\pi}^{n+1}}\int_\R \int_0^\infty \hat{\rho}(t)e^{i\mu t(1-r)}\chi(r)r^{n-1}\left(\int_{S_y^*M}e^{i\mu r\langle \exp_y^{-1}(x),\w \rangle_{g_y}}D^Q(t,y,r\mu\w)d\w\right) dr dt,
\end{align}
where $\chi\in C_c^\infty(\R)$ is a cut-off function that is identically $1$ near $r=1$ and vanishes for $r\not \in [\tfrac{1}{2},\tfrac{3}{2}].$ Indeed, on the support of $1-\chi,$ the operator $L=\frac{1}{i\mu\lr{1-r}}\dell_t$ is well-defined, preserves $e^{i\mu t \lr{1-r}},$ and its adjoint $L^*$ satisfies that for all $k \in \Z^+$
\[\abs{\lr{L^*}^k\lr{r^{n-1}(1-\chi(r))\hat{\rho}(t) \int_{S_y^*M}e^{i\mu r\langle \exp_y^{-1}(x),\w \rangle_{g_y}} D^Q(t,y,r\mu\w)d \w\big.}}\leq \lr{1+\mu}^{-k}\cdot c_k\]
for some $c_k>0.$ Define
\[S^Q(t,y,\xi) := q_0(y,\xi)+D_{-1}^Q(t,y,\xi)\] 
to be the two leading terms of $D^Q.$ Since $D^Q-S^Q\in S^{-2},$ up to a $O(\mu^{n-3})$ error, we have
\begin{align}
&\Theta(x,y)^{\frac{1}{2}} \cdot \partial_\mu (\rho * EQ^*)(x,y,\mu)= \notag\\
&\frac{\mu^n}{\lr{2\pi}^{n+1}}\int_\R \int_0^\infty \hat{\rho}(t)e^{i\mu t(1-r)}\chi(r)r^{n-1}\left(\int_{S_y^*M}e^{i\mu r\langle \exp_y^{-1}(x),\w \rangle_{g_y}}S^Q(t,y,r\mu\w)d\w\right) dr dt. \label{E:Compact Int 3}
\end{align}
According to \cite[Theorem 1.2.1]{Sog} there exist smooth functions $a_{\pm}, b_{\pm} \in C^{\infty}(M \times \R^n)$ such that for all $(y,\eta) \in M \times T_y^*M$
 \begin{align}
\int_{S_y^*M}e^{i \langle\eta,\omega\rangle_{g_y}} S^Q(t,y,\mu r\w)\frac{d\w}{\sqrt{|g_y|}}&= \sum_{\pm}e^{\pm i|\eta|_{g_y}}\lr{a_\pm(y,\eta)+r^{-1}\mu^{-1}\cdot b_{\pm}(t,y,\eta)},\label{E:FT q0}
\end{align}
and
\begin{align}
|\partial_\eta^\alpha a_{\pm}(y, \eta)|  &\leq C_\alpha (1+|\eta|_{g_y})^{-\frac{n-1}{2}-|\alpha|},\label{E: a_pm} \\
|\partial_t^\beta\partial_\eta^\alpha b_{\pm}(t,y, \eta)| & \leq C_{\alpha,\beta} (1+|\eta|_{g_y})^{-\frac{n-1}{2}-|\alpha|-1},\label{E: b_pm}
\end{align}
for all multi-indices $\alpha\geq 0$ and $\beta\geq 0$, and for some $C_\alpha,C_{\alpha, \beta}>0$ independent of $t,y,$ and $\eta$. Hence, \eqref{E:Compact Int 3} equals
\begin{align}
\frac{\mu^n}{(2\pi)^{n+1}} \sum_{\pm} \int_\R \int_0^\infty e^{i \mu \psi_\pm(t,r,x,y)}g_{\pm}(t,r,x,y,\mu) dr dt, \label{E:Compact Int 2}
\end{align}
where $\psi_\pm(t,r,x,y)=t(1-r) \pm r\dist(x,y)$ and
$$g_{\pm}(t,r,x,y,\mu)=\frac{1}{(2\pi)^{n}}r^{n-1}\chi(r) \hat \rho(t)\lr{ a_\pm(y, r\mu \exp_y^{-1}(x))+r^{-1}\mu^{-1}b_{\pm}(t,y, r\mu \exp_y^{-1}(x))}.$$
Note that the critical points of $\psi_\pm$ are $(t_c^{\pm},r_c^{\pm})=(\pm \dist(x,y),1)$ and that 
\[\det ( \text{Hess} \,\psi_{\pm}(t_c^\pm,r_c^\pm,x,y))=1.\]
Hence, we apply the method of stationary phase to get that \eqref{E:Compact Int 2} is 
 \begin{align}
&\mu^{n-1}  e^{\pm i\mu \dist(x,y)}\sum_{\pm} \Big(g_{\pm}(t_c^{\pm}, r_c^\pm, x,y,\mu)- i\mu^{-1}\dell_r\dell_t g_\pm(t_c^{\pm}, r_c^\pm ,x,y,\mu)\Big) \label{E: SP eq1}\\
&\qquad+O \Big(\mu^{n-3}\sup_{(t,r) \in \supp (g_\pm)}\sup_{\alpha+\beta \leq 7}\abs{\dell_t^{\alpha}\dell_r^{\beta}g_\pm(t,r,x,y,\mu)}\Big). \label {E: SP eq2}
\end{align}
We take $7$ derivatives in the last term since in stationary phase with a quadratic phase over $\R^k,$ the remainder after the first $N$ terms is bounded by $k+1+2N$ derivatives of the amplitude. Note that $\partial_t \hat \rho(t)=0$ for $t=\pm \dist(x,y).$ Hence, since $a_{\pm}$ are independent of $t,$ we have
\[ i\mu^{-1}\dell_r\dell_t g_\pm(t_c^{\pm}, r_c^\pm ,x,y,\mu)=O\lr{\mu^{-2}}.\]
Moreover, by \eqref{E: a_pm} and \eqref{E: b_pm}, the derivatives of $g$ in $t$ and $r$ are uniformly bounded. Hence, 
\begin{align}\label{E: error from q_0}
&\frac{\mu^{n-1}}{\lr{2\pi}^n}\int_{S_y^*M}e^{i \langle\exp_y^{-1}(x),\omega\rangle_{g_y}} \lr{q_0(y,\omega) +\mu^{-1}D_{-1}^Q(\dist(x,y),y,\w)}\frac{d\w}{\sqrt{|g_y|}}\\
&\qquad +O(\mu^{n-3}).
\end{align}
Taylor expanding $D_{-1}^Q(\dist(x,y),y, \w)=D_{-1}^Q(0,y,\w)+O(\dist(x,y))$ and recalling \eqref{E:D1} and \eqref{E:D2} completes the proof. 
\end{proof}

\subsection{Proof of Proposition \ref{P:Smoothed Est 1}}\label{S Est 1}
Proposition \ref{P:Smoothed Est 1} follows by integrating \eqref{E:Smoothed Est 1} with respect to $\mu$ from $0$ to $\leb$ applied to $Q=Id$. We have
\begin{equation}
\rho*E(x,y,\leb) =\int_0^\leb \frac{\mu^{n-1}}{(2\pi)^n\Theta(x,y)^{\frac{1}{2}}}\left(\int_{S_y^*M}e^{i \mu\langle \exp_y^{-1}(x),\w\rangle_{g_y}} \frac{d\omega}{\sqrt{|g_y|}} \right)d\mu+ \int_0^\leb W(x,y,\mu) d\mu.\label{E:Smoothed Est Main}
\end{equation}
Changing coordinates to $\xi=\mu \w$ we find 
\begin{equation}
\rho*E(x,y,\leb) =\frac{\leb^n}{\lr{2\pi}^n \Theta(x,y)^{1/2}}\int_{\abs{\xi}_{g_y}<1}e^{i \leb\langle \exp_y^{-1}(x),\xi\rangle_{g_y}} \frac{d\xi}{\sqrt{|g_y|}}+ \int_0^\leb W(x,y,\mu) d\mu.\label{E:Smoothed Est Main2}
\end{equation}
Note that 
\[\Theta(x,y)^{-1/2}=1+O(\dist(x,y)^2)\]
and 
\[\frac{\exp_y^{-1}(x)}{i\leb\dist(x,y)^2}\nabla_\xi e^{i \leb\langle \exp_y^{-1}(x),\xi\rangle_{g_y}}=e^{i \leb\langle \exp_y^{-1}(x),\xi\rangle_{g_y}}.\]
Therefore, we may integrate by parts once in \eqref{E:Smoothed Est Main2} to obtain 
\begin{align*}
\rho*E(x,y,\leb)& =\frac{\leb^n}{\lr{2\pi}^n}\int_{\abs{\xi}_{g_y}<1}e^{i \leb\langle \exp_y^{-1}(x),\xi\rangle_{g_y}} \frac{d\xi}{\sqrt{|g_y|}}+ \int_0^\leb W(x,y,\mu) d\mu\\
&+ O\left(\dist(x,y)\leb^{n-1} \int_{\abs{\xi}_{g_y}=1}e^{i \leb\langle \exp_y^{-1}(x),\w\rangle_{g_y}} d\w \right)
\end{align*}
Since 
\[\sup_{\dist(x,y)<\inj(M,g)}\abs{\dist(x,y)\int_{\abs{\xi}_{g_y}=1}e^{i \leb\langle \exp_y^{-1}(x),\w\rangle_{g_y}} d\w}=o(1)\] 
as $\leb\gives \infty$, we find that 
\begin{align*}
\rho*E(x,y,\leb)& =\frac{\leb^n}{\lr{2\pi}^n}\int_{\abs{\xi}_{g_y}<1}e^{i \leb\langle \exp_y^{-1}(x),\xi\rangle_{g_y}} \frac{d\xi}{\sqrt{|g_y|}}+ \int_0^\leb W(x,y,\mu) d\mu+o(\leb^{n-1})
\end{align*}
By \eqref{E:W long range}, we have
\[\sup_{x,y\in B(x_0,\inj(M,g)/2)}\abs{\int_0^\leb W(x,y,\mu)d\mu}\leq c\dist(x,y)\leb^{n-1}+C \leb^{n-2},\]
for some $c,C>0$ as claimed. 
\qed

\section{Smooth vs rough projector: proof of Proposition \ref{P:Smoothed Est 2}}\label{S Est 2}
Let $x_0\in M$ be a non self-focal point and fix $\ep>0$. The proof of  Proposition \ref{P:Smoothed Est 2} amounts to show that there exists $c>0$ so that for all $\ep>0$ there is an open neighborhood $\mathcal U_\ep$ of $x_0$ and a positive constant $c_{\ep}$ with
\begin{equation}
\sup_{x,y\in \mathcal U_\ep}\abs{E_\leb(x,y)-\rho*E_\leb(x,y)}\leq c\, \ep \leb^{n-1}+c_{\ep}\leb^{n-2}\label{E:NTS0}
\end{equation}
 for all $\leb\geq 1$. 
 It is at this point that the assumption that $x_0$ is a non self-focal point is needed. In \S \ref{S:Microlocal Identity} we construct a partition of the Identity operator localized to $x_0$. We use  such partition to split $\abs{E_\leb(x,y)-\rho*E_\leb(x,y)}$ into different pieces, each of which we shall control using two types of Tauberian Theorems described in  \S  \ref{S:Tauberian}. We conclude this section presenting  the proof of Proposition \ref{P:Smoothed Est 2} in \S \ref{S: Proof of Prop Est 2}.

 To ease the notation, we will write
\[E(x,y,\leb):=E_\leb(x,y).\]
To prove \eqref{E:NTS0}, we use the operators $B_\ep,C_\ep$ and the function $\psi_\ep$ constructed in Lemma \ref{P:Microlocal Identity}. We set 
\begin{align}
\alpha_\ep(x,y,\leb)&:= EC_\ep^*(x,y,\leb) +\tfrac{1}{2}\lr{E(x,x,\leb)+C_\ep EC_\ep^*(y,y,\leb) \big.}, \label{alpha} \\
\beta_\ep(x,y,\leb)&:=\rho*EC_\ep^*(x,y,\leb) +\tfrac{1}{2}\lr{E(x,x,\leb)+C_\ep EC_\ep^*(y,y,\leb)\big.},\label{beta}
\end{align}
where $x$ and $y$ are any two points in $M$. Note that 
$$|\alpha_\ep(x,y,\leb)-\beta_\ep(x,y,\leb)|=|EC_\ep^*(x,y,\leb)-\rho * EC_\ep^*(x,y,\leb)|.$$
In addition, observe that 
$$\alpha_\ep(x,y,\leb):= \tfrac{1}{2}\sum_{\leb_j\leq \leb} \left[\varphi_j(x)+(C_\ep \varphi_j)(y)\right]^2,$$
and so $\alpha_\ep(x,y,\leb)$ is an increasing function of $\leb$ for any $x,y$ fixed. We also set
\begin{equation}\label{g}
g_\ep(x,y, \leb):=EB_\ep^*(x,y,\leb)-\rho * EB_\ep^*(x,y,\leb).
\end{equation}

Since $B_\ep+C_\ep=\psi_\ep^2$ and $\psi_\ep=1$ in a neighborhood of $x_0$, relation \eqref{E:NTS0} would hold if we prove that there exist positive constants $c$ and $c_{\ep}$, with $c$ independent of $\ep$, and a neighborhood $\mathcal U_\ep$ of $x_0$ such that for all $\leb \geq 1$
\begin{align}
\label{E:NTS1} &\sup_{x,y\in \mathcal U_\ep}\abs{\alpha_\ep(x,y,\leb)-\beta_\ep(x,y,\leb)}\leq c\, \ep \leb^{n-1}+c_{\ep}\leb^{n-2},\\
\label{E:NTS2}&\sup_{x,y\in \mathcal U_\ep}\abs{g_\ep(x,y,\leb)}\leq c\,\ep \leb^{n-1}+c_{\ep}\leb^{n-2}.
\end{align}

\subsection{Tauberian Theorems}\label{S:Tauberian}
To control $\abs{\alpha_\ep(x,y,\leb)-\beta_\ep(x,y,\leb)}$ and $\abs{g_\ep(x,y,\leb)}$ we use two different Tauberian-type theorems. To state the first one, fix a positive function $\phi\in \mathcal S(\R)$ so that $\supp \hat{\phi}\subseteq (-1,1)$ and $\hat{\phi}(0)=1.$ We have written $\hat{f}$ for the Fourier transform of $f.$ Define for each $a>0$
\begin{equation}
  \label{E:Phia Def}
  \phi_a(\leb):=\tfrac{1}{a}\phi\lr{\tfrac{\leb}{a}},
\end{equation}
so that $\hat{\phi}_a(t)=\hat{\phi}(at).$

\begin{Lem}[Tauberian Theorem for monotone functions] \label{L: FTT 1} 
Let $\alpha$ be an increasing temperate function with $\alpha(0)=0$ and let $\beta$ be a function of locally bounded variation with $\beta(0)=0.$ Suppose further that there exist $M_0>0$, $a>0$ and a constant $c_a$ so that the following two conditions hold:
\begin{enumerate}
\item[(a)] There exists $m \in \mathbb N$ so that
  \begin{align*}
   \int_{\mu-a}^{\mu+a} \abs{d\beta}&\leq a M_0 (1+\abs{\mu})^{m-1} + c_a \abs{\mu}^{m-2} \qquad \forall \mu\geq 0.
  \end{align*}

\item[(b)] There exist $\kappa \in \Z \backslash \{-1\}$ with $\kappa \leq m-1$, and $M_a>0$ so that 
\begin{align*}
  \abs{\lr{d\alpha-d\beta}*\phi_a(\mu)}&\leq M_a \lr{1+\abs{\mu}}^{\kappa}\qquad \forall \mu\geq 0.
\end{align*}

\end{enumerate}
Then, there exists $c>0$ depending only on $\phi$ such that
  \begin{equation}
\abs{\alpha(\mu)-\beta(\mu)}\leq c \lr{a\,M_0\abs{\mu}^{m-1}+ c_a \abs{\mu}^{m-2} + M_a\lr{1+\abs{\mu}}^{\kappa +1}},
\end{equation}
for all $ \mu\geq 0.$
\end{Lem}
\begin{proof}
The proof is identical to argument for Lemma 17.5.6 in \cite[Volume 3]{Hor 1-4}. 
\end{proof}

We will also need the following result. 
\begin{Lem}[\cite{Hor} Tauberian Theorem for non-monotone functions]\label{L: FTT 2}
Let $g$ be a piecewise continuous function such that there exists $a>0$ with $\hat{g}(t)\equiv 0$ for $\abs{t}\leq a.$ Suppose further that for all $\mu \in \R$ there exist constants $m\in \mathbb N$ and $c_1,c_2>0$ so that 
\begin{equation}
\abs{g(\mu+s)-g(\mu)}\leq c_1\lr{1+\abs{\mu}}^m + c_2\lr{1+\abs{\mu}}^{m-1}\qquad \forall s\in [0,1].\label{E:g cond}
\end{equation}
Then, there exists a positive constant $c_{m,a}$, depending only on $m$ and $a$, such that for all $\mu$
\[\abs{g(\mu)}\leq c_{m,a}\lr{c_1\lr{1+\abs{\mu}}^m+c_2\lr{1+\abs{\mu}}^{m-1}}.\]
\end{Lem}

\subsection{Proof of Proposition \ref{P:Smoothed Est 2}}\label{S: Proof of Prop Est 2}
As explained above, the proof of Proposition \ref{P:Smoothed Est 2} reduces to establishing relations \eqref{E:NTS1} and \eqref{E:NTS2}.
\subsubsection {Proof of relation \eqref{E:NTS1}}
 We seek to apply Lemma \ref{L: FTT 1} to $\alpha_\ep$ and $\beta_\ep$. Let $a=\ep$, $m=n$ and $\kappa=-2.$ We first verify condition $(a)$. From Remark \ref{remark R_Q} it follows that there exist an open neighborhood $\mathcal U_\ep$ of $x_0$ and  constants $c_1, c_\ep>0$ so that for all $x,y \in  \mathcal U_\ep$ and all $\leb \geq 1$
\begin{align}
\notag \int_{\leb-\ep}^{\leb+\ep} (|\partial_\nu\,E(x,x, \nu)|+|\partial_\nu\,(C_\ep E C_\ep^*)(y,y, \nu)|) \,d\nu &= \sum_{|\leb_j-\leb|\leq\ep}(\varphi_j(x))^2 + (C_\ep\varphi_j(y))^2\\
\label{cond a for C (1)}&\leq c_1 \ep \leb^{n-1}  +c_\ep \leb^{n-2}.
\end{align}
Combining \eqref{cond a for C (1)} with the estimate in Proposition \ref{P:Smoothed Estimate} applied to $Q=C_\ep$, we see that there exist positive constants $M_0$ and $c_\ep$ for which 
\begin{align*}
\sup_{x,y\in \mathcal U_\ep}\int_{\leb-\ep}^{\leb+\ep}\abs{\partial_\nu \beta_\ep(x,y,\nu)}\, d\nu
&\leq {M}_0\ep \leb^{n-1}+c_\ep \leb^{n-2}
\end{align*}
for all  $\leb \geq 1$. It remains to verify condition $(b)$. Note that
$$\partial_\leb \Big(\alpha_\ep(x,y,\cdot)-\beta_\ep(x,y,\cdot)\Big)*\phi_\ep\,(\leb)=\mathcal F_{t\gives \leb}^{-1}\lr{(1-\hat{\rho}(t))\,\hat{\phi}_\ep(t)(U(t)C_\ep^*)(x,y)}(\leb),$$
where $\mathcal F$ is the Fourier transform and $\phi_\ep$ is defined in \eqref{E:Phia Def}. According to Lemma \ref{P:Microlocal Identity}, $U(t)C_\ep^*$ is a smoothing operator for $\tfrac{1}{2}\inj(M,g)<\abs{t}<\frac{1}{\ep}$. 
Hence, since $\supp \hat \phi_\ep  \subset \{t:\, |t|< \frac{1}{\ep} \}$ and  $\supp (1-\hat \rho)  \subset \{t:\, |t|> \tfrac{1}{2} \inj(M,g)\}$,  we find that for each $N$ there are constants $c_{N,\ep}$ depending on $N,\ep$ satisfying
\[\sup_{x,y\in M}\abs{\partial_\leb\lr{\alpha_\ep(x,y,\cdot)-\beta_\ep(x,y,\cdot)\big.}*\phi_\ep\,(\leb)}\leq c_{N,\ep}\lr{1+\abs{\leb}}^{-N}\]
for all $\leb >0$.\\

\subsubsection{ Proof of relation \eqref{E:NTS2}}
 We seek to apply Lemma \ref{L: FTT 2} to $g_\ep.$ First, note that since $g_\ep(x,y,\leb)=EB_\ep^*(x,y,\leb)-\rho * EB_\ep^*(x,y,\leb)$, the function $g_\ep(x,y, \cdot)$ is piecewise continuous in the $\leb$ variable. Next, we check that $\hat{g_\ep}(t)\equiv 0$ in a neighborhood of $t=0.$ We have
$$\partial_\leb g_\ep(x,y,\leb)=\mathcal F_{t\gives \leb}^{-1}\lr{(1-\hat{\rho}(t)) (U(t)B_\ep^*)(x,y)}(\leb).$$
Since $\hat \rho\equiv 1$ on $(-\tfrac{1}{2}\inj(M,g), \tfrac{1}{2}\inj(M,g) )$, it follows that  $\mathcal F_{\leb \gives t}(\partial_\leb g_\ep(x,y,\cdot)) (t) =0$ for $|t| \leq \tfrac{1}{2}\inj(M,g) $. Equivalently,
$$t \cdot \mathcal F_{\leb \gives t}(g_\ep(x,y,\cdot)) (t) =0 \qquad |t| \leq \tfrac{1}{2}\inj(M,g).$$
In addition, we must have $\mathcal F_{\leb \gives t}(g_\ep(x,y,\cdot)) (0) = 0$ for otherwise $g_\ep(x,y, \cdot)$ would include a sum of derivatives of delta functions but this is not possible since $g_\ep(x,y,\cdot)$ is piece-wise continuous. It follows that 
$$ \mathcal F_{\leb \gives t}(g_\ep(x,y,\cdot)) (t) =0 \qquad |t| \leq \tfrac{1}{2}\inj(M,g),$$
as desired. It therefore remains to check that $g_\ep$ satisfies \eqref{E:g cond}. Let $s \in [0,1]$, $\leb \in \R$, and write
\begin{align}\label{decompose g}
g_\ep(x,y,\leb+s)-g_\ep(x,y,\leb)=& EB_\ep^*(x,y, \leb+s)-EB_\ep^*(x,y, \leb) \notag\\
&+\rho*EB_\ep^*(x,y, \leb+s)-\rho*EB_\ep^*(x,y, \leb).
\end{align}
To estimate $EB_\ep^*(x,y, \leb+s)-EB_\ep^*(x,y, \leb)$ we apply the Cauchy Schwartz inequality,
\begin{align*}
EB_\ep^*(x,y, \leb+s)-EB_\ep^*(x,y, \leb)
&= \sum_{\leb \leq \leb_j \leq \leb +s} \varphi_j(x)\, B_\ep \varphi_j(y)\\
&\leq  \Big( \sum_{\leb \leq \leb_j \leq \leb +s} (\varphi_j(x))^2\Big)^{\frac{1}{2}} \Big( \sum_{\leb \leq \leb_j \leq \leb +s} (B_\ep\varphi_j(y))^2\Big)^{\frac{1}{2}}
\end{align*}

Applying Remark \ref{remark R_Q} to $Q=Id$ and $Q=B_\ep$, there exist $\mathcal U_\ep$ open neighborhood of $x_0$ and constants $c, c_{\ep}>0$ making 
\begin{equation}\label{decompose g 1}
|EB_\ep^*(x,y, \leb+s)-EB_\ep^*(x,y, \leb)| \leq c \, \ep\, \leb^{n-1} + c_{ \ep} \leb^{n-2}
\end{equation}
for all $\leb \geq 1$,  $s \in [0,1]$, and  $x,y \in  \mathcal U_\ep $. The $\ep$ factor is due to the fact that $\|b_0\|_1<\ep$.  

To estimate $\rho*EB_\ep^*(x,y, \leb+s)-\rho*EB_\ep^*(x,y, \leb)$ we apply Proposition \ref{P:Smoothed Estimate} to the operator $Q=B_\ep$. Since there exists $\tilde c>0$ with 
$$|\partial_\leb \rho*EB_\ep^*(x,y, \leb)| \leq \tilde c \,  (\|b_0\|_1 \leb^{n-1} + \leb^{n-2}) \qquad \forall \leb \geq 1$$ and
 $\|b_0\|_1 \leq \ep$,
we get (after possibly enlarging $c$ and $c_{\ep}$) that
\begin{equation}\label{decompose g 2}
|\rho*EB_\ep^*(x,y, \leb+s)-\rho*EB_\ep^*(x,y, \leb)| \leq c\, \ep\, \leb^{n-1} + c_{\ep} \leb^{n-2} \qquad  \forall \leb \geq 1.
\end{equation}

Combining \eqref{decompose g 1} and \eqref{decompose g 2} into \eqref{decompose g} we conclude 
the existence of positive constants $c$ and $c_{\ep}$ so that
$$|g_\ep(x,y,\leb+s)-g_\ep(x,y,\leb)| \leq c\, \ep\, \leb^{n-1} + c_{\ep} \leb^{n-2}$$
for all $\leb \geq 1$ and $s \in [0,1]$ as desired.
Applying Lemma \ref{L: FTT 2} with $m=n$, $a=\tfrac{1}{2}\inj(M,g)$ proves \eqref{E:NTS2} .

\section{Proof of Theorems \ref{T:LWL aperiodic} -  \ref{T:Dist Function}}\label{S:Applications}
\subsection{Proof of Theorem \ref{T:LWL aperiodic}}\label{S:Aperiodic Pf}
Suppose that $(M,g)$ is a smooth, compact, Riemannian manifold, with no boundary. Let $K\subseteq M\x M$ be a compact set satisfying that any pair of points in it are mutually non-focal. We aim to show that there exists $c>0$ so that for every $\ep>0$ there are constants $\leb_\ep>0$ and $c_\ep>0$ so that 
$$\sup_{\lr{x,y} \in K} |R(x,y, \leb)| \leq c\, \ep  \leb^{n-1}+ c_\ep \leb^{n-2}$$
for all $\leb>\leb_\ep$. Fix $\ep>0$ and write $\Delta \subseteq M\x M$ for the diagonal. Define 
\[\twiddle{K}=K\cap \Delta.\] 
By Equation \eqref{E:Main Goal bis}, there exists $\leb_\ep>0,$ a finite collection $\set{x_j, ~j=1,\ldots, N_\ep},$ and open neighborhoods $\mathcal U_\ep^{x_j}$ of $x_j$ so that 
\[\twiddle{K}\subseteq \bigcup_j \mathcal U_\ep^{x_j}\x \mathcal U_\ep^{x_j}\] 
and
\begin{equation}\label{E: Safarov complement}
\sup_{x,y \in \mathcal U_{\ep}^{x_j} } |R(x,y, \leb)| \leq  c\, \ep \leb^{n-1}+ c_\ep \leb^{n-2}
\end{equation} 
for all $\leb> \leb_\ep.$ Define
$$K_{\ep}:=K ~~\backslash ~~\Union_j \mathcal U_\ep^{x_j}\x \mathcal U_\ep^{x_j}.$$
Safarov proved in \cite[Theorem 3.3]{Saf}, under the mutually non-focal assumption,  that
\begin{equation}\label{E: Safarov}
\sup_{(x,y) \in K_\ep} |R(x,y, \leb)| = o_\ep( \leb^{n-1}).
\end{equation}
Combining \eqref{E: Safarov complement} and \eqref{E: Safarov} completes the proof.
\qed

\subsection{Proof of Theorem \ref{Embedding}}\label{S:Pf Embedding}The injectivity of the maps $\Psi_{_{(\leb, \leb+1]}}:M \to \R^{m_\leb}$ for $\leb$ large enough is implied by the existence of positive constants $c_1,c_2, r_0$ and $\leb_{r_0}$ so that if $\leb>\leb_{r_0},$ then 
\begin{equation}
 \inf_{x,y:\,\leb \dist(x,y)\geq r_0}\Dist_\leb^2(x,y)>c_1\label{E:NTS Embedding 2}
\end{equation}
and
\begin{equation}
\inf_{x,y:\,\leb \dist(x,y)<r_0}\frac{\Dist_\leb^2(x,y)}{\leb^2 \dist(x,y)^2}>c_2.\label{E:NTS Embedding 1}
\end{equation}
We first prove \eqref{E:NTS Embedding 2}. By Theorem \ref{T:LWL aperiodic}, for all $x,y \in M,$
\begin{align}
\Dist_\leb^2(x,y)= f(\leb \dist(x,y)) ~~+~~ \twiddle{R}(x,y,\leb),\label{E:Embedding Est}
\end{align}
where $\sup_{x,y \in M} |\twiddle R(x,y,\leb)|=o(1)$ and $f:[0,+\infty)\to \R$ is the function
\[f(r):=\int_{S^{n-1}}1-e^{ir\w_1}d\w.\]
Observe that $f(r)\geq 0$ with $f(r)=0$ only if $r=0.$ Moreover, 
\begin{equation}
f(r)=\sigma_n+O(r^{-\frac{n-1}{2}}) \;\; \text{as}\; \; r \to \infty \qquad \text{and}\qquad f(r)=r^2\cdot \twiddle{f}(r)\label{E:f rlns}
\end{equation}
for some smooth and positive function $\twiddle{f}$, where $\sigma_n$ is the volume of $S^{n-1}$. According to the first relation in \eqref{E:f rlns}, we may choose $r_0>0$ so that 
\begin{equation}\label{E:Embedding Est 2}
\text{if}\quad \leb \dist(x,y) \geq r_0 \quad \text{then}\quad \abs{f(\leb \dist(x,y))- \sigma_n}\leq \frac{\sigma_n}{4}.
\end{equation}
Moreover, by Theorem \ref{T:LWL aperiodic} we may choose $\leb_{r_0}$ so that if $\leb>\leb_{r_0},$ then
\begin{equation}\label{E:Embedding Est 3}
\sup_{x,y \in M} \abs{\twiddle R(x,y, \leb)}\leq \frac{\sigma_n}{4}.
\end{equation}
Combining \eqref{E:Embedding Est} , \eqref{E:Embedding Est 2}  and \eqref{E:Embedding Est 3}, we proved that for all $\leb> \leb_{r_0}$ and all $x,y \in M$ with $\leb \dist(x,y) \geq r_0$
$$\Dist^2_\leb(x,y) \geq \frac{\sigma_n}{2},$$ 
as desired. To verify \eqref{E:NTS Embedding 1}, write as above
\[\Dist^2_\leb(x,y)=\frac{(2\pi)^n}{2 \leb^{n-1}} (E_{_{(\leb,\leb+1]}}(x,x)+E_{_{(\leb,\leb+1]}}(y,y)-2E_{_{(\leb,\leb+1]}}(x,y)),\]
and note that the first derivatives of $\Dist_\leb^2(x,y)$ in $x$ and $y$ all vanish when $x=y.$ Moreover, by \cite[Proposition 2.3]{Zel2}, we have that the Hessian of $E_{_{(\leb,\leb+1]}}$ may be written as 
\[d_x\otimes d_y\big|_{x=y} E_{_{(\leb, \leb+1]}}(x,y)=C_n \leb^{n+1}  g_x +o(\leb^{n+1}),\]
where $g_x$ is the metric $g$ on $T_xM$ and Equation (1.2) in \cite{Pot} shows that
\[C_n=\frac{\sigma_n}{n(2\pi)^n}.\]
Therefore, applying Taylor's Theorem, we have that there exists ${C_0}>0$ for which
\begin{equation}
\abs{\frac{\Dist_\leb^2(x,y)}{\leb^2 \dist^2(x,y)}- \frac{\sigma_n}{2n}}\leq {C_0}\cdot \leb \dist(x,y).\label{E:Very Near-Diag}
\end{equation}
The extra factor of $\leb$ on the right hand side of \eqref{E:Very Near-Diag} comes from the fact that 
\[\sup_{\abs{\alpha}=3} \abs{\partial_x^{\alpha}|_{x=y} E_{_{(\leb,\leb+1]}}(x,y)}=O(m_\leb \, \leb^3),\]
which is proved for example in \cite[Equation (2.7)]{Bin}. Equation \eqref{E:Very Near-Diag} shows that
\begin{equation*}
\inf_{\leb \dist(x,y)<\frac{\sigma_n}{4 n{C_0}}}\frac{\Dist_\leb^2(x,y)}{\leb^2 d^2_g(x,y)}\geq \frac{\sigma_n}{2n}>0.
\end{equation*}
If $r_0 \leq \frac{\sigma_n}{4 n{C_0}}$, then the claim \eqref{E:NTS Embedding 1} follows. Otherwise, it remains to show that there exists $c_2>0$ with
\begin{equation}
\inf_{\frac{\sigma_n}{4n C_0} \leq \leb \dist(x,y)<r_0}\frac{\Dist_\leb^2(x,y)}{\leb^2 d^2_g(x,y)}>c_2
\end{equation}
for all $\leb$ sufficiently large. Theorem \ref{T:LWL aperiodic} shows that, after possibly enlarging $\leb_{r_0},$ we have
\[\sup_{x,y \in M}\abs{\twiddle{R}(x,y,\leb)}\leq \Big(\frac{\sigma_n}{4 n C_0}\Big)^2\inf_{r<r_0}\tilde f(r)\]
for all $\leb>\leb_{r_0}$. Then, the second relation in \eqref{E:f rlns} combined with \eqref{E:Embedding Est} yields that for all $\leb>\leb_{r_0}$
\[\inf_{\frac{\sigma_n}{4n C_0} \leq \leb \dist(x,y)< r_0}\Dist_\leb^2(x,y)\geq \lr{\frac{\sigma_n}{4nC_0}}^2\inf_{r<r_0} \tilde f(r)>0.\]
This completes the proof of \eqref{E:NTS Embedding 1}. 

\subsection{Proof of Theorem \ref{T:Dist Function}}\label{S:Dist Function Pf}
By \eqref{E:dist def} and Theorem \ref{T:LWL aperiodic} we have that 
\[\sup_{x,y\in M}\abs{\Dist_\leb^2(x,y)-\int_{S^{n-1}}\lr{1-e^{i\leb \dist(x,y)\w_1}}d\w}=o\lr{1}\]
as $\leb \gives \infty.$ Combing this with
\[\frac{1}{\leb^2\dist(x,y)^2}\int_{S^{n-1}}\lr{1-e^{i\leb \dist(x,y)\w_1}}d\w=\frac{\sigma_n}{2n}+O(\leb^2 \dist^2(x,y)).\]
and with Equation \eqref{E:Very Near-Diag} completes the proof. 
\qed

\end{document}